\documentclass[11pt]{amsart}
\usepackage{amssymb,latexsym}
\usepackage{enumerate}
\usepackage[T1]{fontenc}
\usepackage{mathrsfs}
\usepackage{amsmath}
\usepackage{yhmath}
\makeatletter
\@namedef{subjclassname@2010}{%
  \textup{2010} Mathematics Subject Classification}
\makeatother
\makeindex

\newtheorem{theorem}{Theorem}[subsection]
\newtheorem{cor}[theorem]{Corollary}

\newtheorem{prop}[theorem]{Proposition}

\theoremstyle{definition}
\newtheorem{df}[theorem]{Definition}
\newtheorem{rem}[theorem]{Remark}
\newtheorem{ex}[theorem]{Example}

\numberwithin{theorem}{section}

\newcommand{\mb}{\mathbb}
\newcommand{\mc}{\mathcal}
\newcommand{\mf}{\mathfrak}
\newcommand{\ms}{\mathscr}
\newcommand{\s}{\subset}

\begin{document}
\title[A new Hartogs...]{A new Hartogs type extension results for the cross-like objects}

\author{Arkadiusz Lewandowski}

\address{Institute of Mathematics\\ Faculty of Mathematics and Computer Science\\ Jagiellonian University\\ {\L}ojasiewicza 6,
30-348 Kraków, Poland}
\email{arkadiuslewandowski@gmail.com; Arkadiusz.Lewandowski@im.uj.edu.pl}
\thanks{Project operated within the Foundation for Polish Science
IPP Programme ``Geometry and Topology in Physical Models''
co-financed by the EU European Regional Development Fund,
Operational Program Innovative Economy 2007-2013.}
\date{}

\begin{abstract}
We discuss the problem of the existence of envelopes of holomorphy of the $\mathscr{A}$-crosses, which leads us to the far-reaching generalizations of the famous Hartogs theorem. We also take under consideration the issue of the existence of ``nice'' general descriptions of envelopes of holomorphy of the cross-like objects in terms of the relative extremal function, which seems to be very natural in the light of the extension theory of separately holomorphic functions on classical crosses and $(N,k)$-crosses. 
\end{abstract}

\maketitle

\section{Introduction}
The celebrated Hartogs theorem (\cite{H1}) states that every separately holomorphic function in several complex variables is necessarily holomorphic as a function of all variables. This over one hundred years old deep result has been generalized in many directions. One of them leads to the theory of the extension of separately holomorphic functions defined on the so-called crosses (\cite{A1}). The newest results in this area concern the $(N,k)$-crosses (\cite{J3}) and the generalized $(N,k)$-crosses (\cite{L1},\cite{Z1}).\\ 
\indent When the cross-like objects are concerned, two basic questions appear in a~natural way. The first question is, whether all separately holomorphic functions defined on such objects extend to some open neighborhood of them. If the answer is positive, we can investigate the envelope of holomorphy of that neighborhood. Therefore, the problem of finding the \emph{envelope of holomorphy of a cross-like object} is well-posed. The second question asks whether we can find a nice description of the envelope of holomorphy of the cross-like object (for instance, in terms of the relative extremal function. Recall that for a subset $A$ of an open set $D\s\mb{C}^n$ the \emph{relative extremal function of $A$ with respect to $D$}, $\boldsymbol{h}^{\star}_{A,D}$ is defined as the upper semicontinuous regularization of the function
$$
\boldsymbol{h}_{A,D}:=\sup\{u:u {\rm\ plurisubharmonic\ on\ }D,u\leq 1,u|_A\leq 0\}).
$$ 
The problem is important, since the shape of the envelopes of holomorphy should be known for further possible applications in problems of analytic extensions.\\
\indent It is known that in the class of the classical crosses as well as the $(N,k)$-crosses all the above-mentioned issues have satisfying solutions (see \cite{J3}). However, let us consider the following example.
\begin{ex}[For details see Example \ref{nine}]\label{starting}Let $N=4$, let $D_j\s\mb{C}^{n_j}$ be pseudoconvex domains, and let $\varnothing\neq A_j\s D_j$ be compact, locally pluriregular (i.e. $\boldsymbol{h}^{\star}_{A_j,U}(a):=\boldsymbol{h}^{\star}_{A_j\cap U,U}(a)=0$ for any point $a\in A_j$ and for any open neighborhood $U$ of $a$), and locally $L$-regular (that is, regular in the sense of the Siciak-Zahariuta extremal function, see \cite{Si1}), $j=1,\ldots,4$. Consider the set 
$$
\mathbf{Q}:=
\begin{array}{l}
(A_1\times D_2\times D_3\times D_4)\cup\\
(D_1\times A_2\times A_3\times D_4)\cup\\
(D_1\times A_2\times D_3\times A_4)\cup\\
(D_1\times D_2\times A_3\times A_4).
\end{array}
$$
Take an $\mc{F},$ the family of all functions $f:\mathbf{Q}\rightarrow\mb{C}$, separately holomorphic in the following sense:\\
\indent$\bullet$\ $f(a_1,\cdot)$ is holomorphic on $D_2\times D_3\times D_4$ for any $a_1\in A_1;$\\
\indent$\bullet$\ $f(\cdot,a_2,a_3,\cdot)$ is holomorphic on $D_1\times D_4$ for any $a_2\in A_2,a_3\in A_3;$\\
\indent$\bullet$\ $f(\cdot,a_2,\cdot,a_4)$ is holomorphic on $D_1\times D_3$ for any $a_2\in A_2,a_4\in A_4$;\\
\indent$\bullet$\ $f(\cdot,a_3,a_4)$ is holomorphic on $D_1\times D_2$ for any $a_3\in A_3, a_4\in A_4.$\\
Observe that $\mathbf{Q}$ is not an $(N,k)$-cross (see \cite{J3}). In particular, it is not a classical cross. It is some cross-like object with branches of different size. It can be shown that there exists a pseudoconvex domain $\widehat{\mathbf{Q}}$ containing $\mathbf{Q}$ and such that each $f\in\mc{F}$ admits a function $\widehat{f},$ holomorphic on $\widehat{\mathbf{Q}}$ and satisfying $\widehat{f}=f$ on $\mathbf{Q},$ i.e. $\widehat{\mathbf{Q}}$ is the envelope of holomorphy of $\mathbf{Q}.$ Moreover, $\widehat{\mathbf{Q}}$ may be described explicitly. Namely, we have the equality
\begin{multline}\label{eqstart}
\widehat{\mathbf{Q}}=\{(z_1,z_2,z_3,z_4)\in D_1\times D_2\times D_3\times D_4:\boldsymbol{h}^{\star}_{A_1,D_1}(z_1)+\\
\frac{1}{2}(\boldsymbol{h}^{\star}_{A_2,D_2}(z_2)+\boldsymbol{h}^{\star}_{A_3,D_3}(z_3)+\boldsymbol{h}^{\star}_{A_4,D_4}(z_4)-1)<1\}.
\end{multline}      
\end{ex}
The above example leads to the concept of the $\mathscr{A}$-crosses (see Definition \ref{quasicross def}), objects more general than the $(N,k)$-crosses: for the first time the cross-like objects admit different sizes of the branches. Therefore, the geometric situation is totally different in comparison with the cross-like objects discussed so far. It seems that the $\mathscr{A}$-crosses might be useful in the investigation of the families of holomorphic functions on pseudoconvex domains described in the spirit of (\ref{eqstart}).\\
\indent We shall show that the envelope of holomorphy of the $\mathscr{A}$-cross exists if we only exclude some ``pathological'' situation. Additionally, it coincides with the envelope of holomorphy of some corresponding classical $2$-fold cross (no matter how complicated the starting $\mathscr{A}$-cross is, cf. Theorem \ref{quasi}). Also, the extension theorem with singularities in the context of $\mathscr{A}$-crosses is delivered (Theorem \ref{quasising}).\\  
\indent Concerning the matter of ``nice'' descritions of envelopes of holomorphy of $\mathscr{A}$-crosses, in Section \ref{sec4} we present how different such descriptions can be, which might be important for further applications.\\
\indent The natural objects treated in this article are Riemann regions (in virtue of Thullen theorem). The interested reader is asked to consult \cite{J1} for a wide exposition of the theory of Riemann regions. In the present paper $\mc{PLP}(X)$ stands for the family of all pluripolar subsets of an arbitrary Riemann region $X$ and $\mc{O}(X)$ is the space of all holomorphic functions on $X$; furthermore, by $\mc{PSH}(X)$ we will denote the family of all plurisubharmonic functions on $X.$ The notion of the \emph{relative extremal functions} and the \emph{local pluriregularity} in the context of Riemann regions goes along the same lines as in the Euclidean case. For a Riemann region over $\mb{C}^n$, say $X$, and its subset $A$, the relative exytemal function of $A$ with respect to $X$ will be abbreviated by $\boldsymbol{h}^{\star}_{A,X}.$  
It is a simple observation that if $A$ is locally pluriregular, then
$\boldsymbol{h}^{\star}_{A,X}\equiv \boldsymbol{h}_{A,X}.$ In that case one could omit the star when using the relative extremal function. However, we shall act according to our aesthetics, and whenever we use the relative extremal function, we use it with the star. For a good background on the topic of the relative extremal function we refer the reader to \cite{J2}.\\
\indent Some parts of the paper were written during the author's stays at the Carl von Ossietzky Universit\"{a}t Oldenburg and at the University of Iceland. The author would like to express his gratitude to Professor Peter Pflug and to Professor Ragnar Sigur{\dh}sson from those institutions for their constant help and inspiring discussions. The author also owes thanks to Professor Marek Jarnicki for his commitment and many valuable remarks.

\section{$\mathscr{A}$-crosses and the main result}
\label{sec3} 
\indent Let $D_j$ be a Riemann domain over $\mb{C}^{n_j}$ and let $\varnothing\neq A_j\s D_j$ for $j=1,\ldots,N,$ $N\geq 2.$\\ 
\indent For $k\in\{1,\ldots,N\}$ let $\mf{I}(N,k):=\{\alpha=(\alpha_1,\ldots,\alpha_N)\in\{0,1\}^N:|\alpha|=k\},$ where  $|\alpha|:=\alpha_1+\ldots+\alpha_N.$\\ 
\indent Put 
\begin{displaymath}
\mathbf{\mc{X}}_{\alpha,j}:=\begin{cases}
D_j, &\text{ if }\alpha_j=1\\
A_j, &\text{ if }\alpha_j=0
\end{cases},\quad \mc{X}_{\alpha}:=\prod_{j=1}^N \mc{X}_{\alpha,j}.
\end{displaymath}
\indent For an $\alpha\in \mf{I}(N,k)$ such that $\alpha_{r_1}=\ldots=\alpha_{r_{k}}=1,\alpha_{i_1}=\ldots=\alpha_{i_{N-k}}=0,$  where $r_1<\ldots<r_k$ and $i_1<\ldots<i_{N-k}$, put
\begin{displaymath}
D_{\alpha}:=\prod_{s=1}^k D_{r_s},\quad A_{\alpha}:=\prod_{s=1}^{N-k} A_{i_s}.
\end{displaymath}
For an $a=(a_1,\ldots,a_N)\in\mc{X}_{\alpha},$ where $\alpha$ is as above, put $$a_{\alpha}^0:=(a_{i_1},\ldots,a_{i_{N-k}})\in A_{\alpha}.$$ 
\indent For every $\alpha\in \mf{I}(N,k)$ and every $a=(a_{i_1},\ldots,a_{i_{N-k}})\in A_{\alpha}$ define the mapping
\begin{displaymath}
\boldsymbol{i}_{a,\alpha}=(\boldsymbol{i}_{a,\alpha,1},\ldots,\boldsymbol{i}_{a,\alpha,N}):D_{\alpha}\rightarrow\mc{X}_{\alpha},
\end{displaymath}
\begin{displaymath}
\boldsymbol{i}_{a,\alpha,j}(z):=\begin{cases}
z_j, &\text{if }\alpha_j=1\\
a_{j}, &\text{if }\alpha_j=0
\end{cases},\quad j=1,\ldots, N,\quad z=(z_{r_1},\ldots,z_{r_k})\in D_{\alpha}
\end{displaymath}
(if $\alpha_j=0,$ then $j\in\{i_1,\ldots,i_{N-k}\}$ and if $\alpha_j=1,$ then $j\in\{r_1,\ldots,r_{k}\}$). \\
\indent In \cite{J3} Jarnicki and Pflug introduced the so-called $(N,k)$-crosses.
\begin{df}\rm
An \emph{$(N,k)$-cross} is defined as
\begin{displaymath}
\mathbf{X}_{N,k}=\mb{X}_{N,k}((A_j,D_j)_{j=1}^N):=\bigcup_{\alpha\in \mf{I}(N,k)}\mc{X}_{\alpha}.
\end{displaymath}
The \emph{envelope of an $(N,k)$-cross} is defined as
\begin{multline*}
\widehat{\mathbf{X}}_{N,k}=\widehat{\mb{X}}_{N,k}((A_j,D_j)_{j=1}^N):=\\\{(z_1,\ldots,z_N)\in D_1\times\ldots\times D_N:\sum_{j=1}^N\boldsymbol{h}^{\star}_{A^{\star}_j,D_j}(z_j)<k\}.
\end{multline*}
\end{df}
In particular, if all $A_j$'s are locally pluriregular, then $\mathbf{X}_{N,k}\s\widehat{\mathbf{X}}_{N,k}.$\\
\indent Observe that for $k=1$ the above definition leads to the classical $N$-fold crosses (\cite{J6}). In the case where $N=2$ sometimes it will be more convenient to use the notation $\mb{X}(A_1,A_2;D_1,D_2):=\mb{X}_{2,1}((A_j,D_j)_{j=1,2}).$
In the sequel we shall intensively use the following
\begin{theorem}[\cite{J3}]\label{cross theorem}
For every function $f,$ separately holomorphic on $\mathbf{X}_{N,k}$ (cf. Definition \ref{sepholo}), there exists an $\widehat{f}\in\mc{O}(\widehat{\mathbf{X}}_{N,k})$ such that $\widehat{f}=f$ on $\mathbf{X}_{N,k}$ and $\widehat{f}(\widehat{\mathbf{X}}_{N,k})\s f(\mathbf{X}_{N,k}).$
\end{theorem}
Example \ref{starting} reveals the need of consideration the more ``flexible'' cross-like objects. We have to emphasize that in the context of the $(N,k)$-crosses the shape of all the $\mc{X}_{\alpha}$'s is the same. Here we shall introduce the $\mathscr{A}$-crosses, where the situation can be completely different.
\begin{df}\rm
Fix a natural number $N\geq 2.$ For any $\varnothing\neq S\s\{1,\ldots,N\}$ we choose some system of multiindices $\alpha(s)^1,\ldots,\alpha(s)^{l_s}\in \mf{I}(N,s),s\in S.$ For $s\in\{1,\ldots, N\}\setminus S$ put $l_s=0.$ In the set $\{\alpha(s)^{r}:s\in S,r=1,\ldots,l_s\}$ consider the lexicographical order and denumarate its elements with respect to this order as $\alpha^1,\ldots,\alpha^{l_1+\ldots+l_N}.$ Build the matrix $\mathscr{A}:=~(\alpha_j^i)_{j=1,\ldots,N,i=1,\ldots,l_1+\ldots+l_N}$ and define the \emph{$\mathscr{A}$-cross} 
\begin{displaymath}
\mathbf{Q}(\mathscr{A})=\mathbf{Q}(\alpha^i_j)=\mb{Q}(\alpha^i_j)((A_j,D_j)_{j=1}^N):=\bigcup_{s\in S}\bigcup_{r=1}^{l_s}\mc{X}_{\alpha(s)^r}.
\end{displaymath}
\indent We say that the sets $\mc{X}_{\alpha(s)^r}$ are \emph{branches} of $\mathbf{Q}(\mathscr{A})$.\\ 
\indent The $\mathscr{A}$-cross is said to be \emph{reduced}, if there is no nontrivial chain (with respect to the lexicographical order) in the set $\{\alpha(s)^r:r=1,\ldots,l_s,s\in S\}$ (i.e. there is no situation where some branch $\mc{X}_{\alpha(s_1)^{r_1}}$ is essentially contained in some another branch $\mc{X}_{\alpha(s_2)^{r_2}}$).
\label{quasicross def}
\end{df}
From now on, without loss of generality, we shall consider only the reduced $\mathscr{A}$-crosses.
\begin{ex}\rm\label{2,3}
For $N=2$ we have the following $\mathscr{A}$-crosses: 
$$
\mathbf{Q}\scriptsize(\!\!\begin{array}{cc}1&0\end{array}\!\!)=D_1\times A_2,$$
$$\mathbf{Q}\scriptsize(\!\!\begin{array}{cc}0&1\end{array}\!\!)=A_1\times D_2,$$
$$\mathbf{Q}\scriptsize\left(\!\!\begin{array}{cc}0&1\\1&0\end{array}\!\!\right)=(A_1\times D_2)\cup(D_1\times A_2)=\mb{X}_{2,1}((A_j,D_j)_{j=1,2}),$$
$$\mathbf{Q}\scriptsize(\!\!\begin{array}{cc}1&1\end{array}\!\!)=D_1\times D_2=\mb{X}_{2,2}((A_j,D_j)_{j=1,2}).$$
For $N=3$ we have the following $\mathscr{A}$-crosses (up to permutations of variables):
$$
\mathbf{Q}\scriptsize(\!\!\begin{array}{ccc}0&1&1\end{array}\!\!)=A_1\times D_2\times D_3,
$$
$$
\mathbf{Q}\scriptsize\left(\!\!\begin{array}{ccc}0&1&1\\1&0&0 \end{array}\!\!\right)=(A_1\times D_2\times D_3)\cup(D_1\times A_2\times A_3)=\mb{X}(A_1,A_2\times A_3;D_1,D_2\times D_3),
$$
\begin{multline*}
\mathbf{Q}\scriptsize\left(\!\!\begin{array}{ccc}0&1&1\\1&0&1 \end{array}\!\!\right)=(A_1\times D_2\times D_3)\cup(D_1\times A_2\times D_3)=\mb{X}_{2,1}((A_j,D_j)_{j=1,2})\times D_3=\\\mb{X}(A_1,A_2\times D_3;D_1,D_2\times D_3),
\end{multline*}
$$
\mathbf{Q}\scriptsize\left(\!\!\begin{array}{ccc}0&1&1\\1&0&1\\1&1&0 \end{array}\!\!\right)=\mb{X}_{3,2}((A_j,D_j)_{j=1,2,3}),
$$
$$
\mathbf{Q}\scriptsize(\!\!\begin{array}{ccc}0&0&1\end{array}\!\!)=(A_1\times A_2\times D_3),
$$
$$
\mathbf{Q}\scriptsize\left(\!\!\begin{array}{ccc}0&0&1\\1&0&0 \end{array}\!\!\right)=(A_1\times A_2\times D_3)\cup(D_1\times A_2\times A_3),
$$
$$
\mathbf{Q}\scriptsize\left(\!\!\begin{array}{ccc}0&0&1\\0&1&0\\1&0&0 \end{array}\!\!\right)=\mb{X}_{3,1}((A_j,D_j)_{j=1,2,3}),
$$
$$
\mathbf{Q}\scriptsize(\!\!\begin{array}{ccc}1&1&1\end{array}\!\!)=(D_1\times D_2\times D_3)=\mb{X}_{3,3}((A_j,D_j)_{j=1,2,3}).
$$
\end{ex}
Observe that in this new notation the set of $(N,k)$-crosses equals the set of those $\mathscr{A}$-crosses, for which the rows of the defining matrix $\mathscr{A}$ are exactly the elements of $\mf{I}(N,k)$ ordered lexicographically. The notion of separate holomorphicity is completely in the spirit of $(N,k)$-crosses (see \cite{J3}).
\begin{df}\rm\label{sepholo}
For a relatively closed set $M\s \mathbf{Q}(\alpha^i_j)$ (we allow $M=\varnothing$ here) we say that a function $f:\mathbf{Q}(\alpha^i_j)\setminus M\rightarrow\mb{C}$ is \emph{separately holomorphic on $\mathbf{Q}(\alpha^i_j)\setminus M$} (written $f\in\mc{O}_s(\mathbf{Q}(\alpha^i_j)\setminus M)$) if for every $s\in\{1,\ldots, N\}$ with nonzero $l_s,$ every $j\in\{1,\ldots, l_s\}$, and for every $a\in A_{\alpha(s)^j}$ the function
\begin{displaymath}
D_{\alpha(s)^j}\setminus M_a\ni z\mapsto f(\boldsymbol{i}_{a,\alpha(s)^j}(z))
\end{displaymath} is holomorphic, where
$$
M_a:=\{z\in D_{\alpha(s)^j}:\boldsymbol{i}_{a,\alpha(s)^j}(z)\in M\}
$$
is the \emph{fiber} of $M$ at $a$.
\end{df}
As it was pointed out in the Introduction, we ask whether any separately holomorphic function on an $\mathscr{A}$-cross $\mathbf{Q}$ can be extended holomorphically to some open neighborhood of $\mathbf{Q}$. The second question is, whether we can effectively describe the envelope of holomorphy of an $\mathscr{A}$-cross. Note that usually $\mathscr{A}$-crosses are not open. However, if any separately holomorphic function on an $\mathscr{A}$-cross $\mathbf{Q}$ extends holomorphically to some open neighborhood of $\mathbf{Q},$ we may speak of envelope of holomorphy of that neighborhood. Therefore, the problem of finding the \emph{envelope of holomorphy of a given $\mathscr{A}$-cross} should not lead to confusions.\\
\indent As the following example shows, we need to avoid some ``pathological'' situations.
\begin{ex}
\rm Take 
\begin{displaymath}
\mathbf{Q}:=\mathbf{Q}\scriptsize\left(\!\!\begin{array}{ccc}0&0&1\\1&0&0 \end{array}\!\!\right)=(A_1\times A_2\times D_3)\cup(D_1\times A_2\times A_3),
\end{displaymath}
 Consider on $\mathbf{Q}$ the function $f(z_1,z_2,z_3):=g(z_2)h(z_1,z_3),$ where $h\neq 0$ is some separately holomorphic function on  $\mathbf{X}:=\mb{X}(A_1,A_3;D_1,D_3)$ (possibly constant) and $g$ is some ``wild'' function on $A_2.$ Then $f$ is separately holomorphic on $\mathbf{Q}.$ It may be, however, very far away from being holomorphically extendible (or just holomorphic, in the case where $A_j=D_j,j=1,2,3$). 
\end{ex}
\begin{ex}\rm
Consider
$$
\mathbf{Q}:=\mathbf{Q}(\!\!\begin{array}{ccc}1&\cdots&1\end{array}\!\!)
=\mb{X}_{N,N}((A_j,D_j)_{j=1}^N)=D_1\times\ldots\times D_N.
$$
Then every separately holomorphic function on $\mathbf{Q}$ is authomatically holomorphic on $\mathbf{Q}.$ If in addition all $D_j$'s are domains of holomorphy, then so is~$\mathbf{Q}.$ 
\end{ex}
The following theorem shows that if we exclude the situation where there exists some $k\in\{1,\ldots,N\}$ such that for any $s\in\{1,\ldots,N\}$ with nonzero $l_s$ and for any $\alpha(s)^r,r=1,\ldots,l_s$ we have $\alpha(s)^r_k=0$ (in other words, we assume the inclusion $\mb{X}_{N,1}((A_j,D_j)_{j=1}^N)\s\mb{Q}(\alpha^i_j)((A_j,D_j)_{j=1}^N)$), then we may describe the envelope of holomorphy of an $\mathscr{A}$-cross as the envelope of holomorphy of some corresponding $2$-fold cross.
\begin{theorem}
Let $D_j$ be a Riemann domain of holomorphy over $\mb{C}^{n_j}$ and let $A_j\s D_j$ be locally pluriregular, $j=1,\ldots, N.$ Put $\mathbf{Q}:=\mb{Q}(\alpha^i_j)((A_j,D_j)_{j=1}^N)$. Assume that $\mb{X}_{N,1}((A_j,D_j)_{j=1}^N)\s\mathbf{Q}$. 
Then there exist an $m_0\in\{1,\ldots,N\},$ a domain of holomorphy $G\s D_1\times\ldots\times D_{m_0-1}\times D_{m_0+1}\times\ldots\times D_N,$ a locally pluriregular subset $B\s G$, and a $2$-fold classical cross $\mathbf{X}=\mb{X}(A_{m_0},B;D_{m_0},G)$ containing $\tau^{-1}(\mathbf{Q}),$ where $\tau$ is a mapping which sends a point $$(z_{m_0},z_1,\ldots,z_{m_0-1},z_{m_0+1},\ldots,z_N)\in D_{m_0}\times D_1\times\ldots\times D_{m_0-1}\times D_{m_0+1}\times\ldots\times D_N$$ to the point $(z_1,\ldots,z_N)\in D_1\times\ldots\times D_N,$ such that for every function $f\in\mc{F}:=\mc{O}_s(\mathbf{Q})$ there exists a unique function $\widehat{f}\in\mc{O}(\widehat{\mathbf{X}})$ with $\widehat{f}=f\circ\tau$ on $\tau^{-1}(\mathbf{Q}),$ i.e. $\widehat{\mathbf{Q}}:=\tau(\widehat{\mathbf{X}})$ is the envelope of holomorphy of $\mathbf{Q}.$
\label{quasi}
\end{theorem}
\begin{proof}
\indent We start the induction on $N.$ For $N=2$ the conclusion is obvious (see Example \ref{2,3}). Suppose now that the conclusion holds true for $N-1$ with some $N\geq 3$ and consider $\mathbf{Q}(\alpha^i_j).$\\
\indent Let $\mathfrak{K}$ be the set of those $k\in\{1,\ldots,N\}$ such that for any $s\in\{1,\ldots,N\}$ with non-zero $l_s$ and for any $\alpha(s)^r,r=1,\ldots,l_s$ we have $\alpha(s)^r_{k}=1.$ Two cases have to be considered.\\
\indent\textit{Case 1.} There is some $k_0\in\mathfrak{K}$.\\
\indent Observe that if $\mathfrak{K}=\{1,\ldots, N\},$ then $\mathbf{Q}=D_1\times\ldots\times D_N$ and we take $m_0=1,G=B=D_2\times\ldots\times D_N,\mathbf{X}=\widehat{\mathbf{X}}=D_1\times\ldots\times D_N.$\\
\indent Therefore, we may assume that there is some $m_0\in\{1,\ldots,N\}\setminus\mathfrak{K}.$\\
\indent To simplify the notation we assume that $k_0=N\in\mathfrak{K}$ (the proof in~the~other cases goes along the same lines). Then $$\mathbf{Q}(\alpha^i_j)=\mathbf{Q}'\times D_N,$$ 
where $\mathbf{Q}'$ is some ``smaller'' $\mathscr{A}$-cross, with the defining matrix of dimension $(l'_1+\ldots+l'_{(N-1)})\times(N-1)$.
By the inductive assumption and Terada's theorem (\cite{T1}), for every function from $\mc{F}$, the function $f\circ\tau$ extends holomorphically to $\widehat{\mb{X}}(A_{m_0},B';D_{m_0},G')\times D_N$ with some domain of holomorphy $G'\s D_1\times\ldots\times D_{m_0-1}\times D_{m_0+1}\times\ldots\times D_{N-1}$ and locally pluriregular set $B'\s G'.$ It is left to observe that 
$$\widehat{\mb{X}}(A_{m_0},B';D_{m_0},G')\times D_N=\widehat{\mb{X}}(A_{m_0},B'\times D_N;D_{m_0},G'\times D_N).$$
\indent\textit{Case 2.} The set $\mathfrak{K}$ is empty.\\
\indent We use the following notation: for any $s$ with non-zero $l_s$ and any number $r\in\{1,\ldots, l_s\}$ write the multi-index $\alpha(s)^r$ as $(\alpha(s)^r_1,\alpha(s)^r_0),$ where $\alpha(s)^r_0$ is~a~suitable multi-index from $\{0,1\}^{N-1}.$\\
\indent Fix a point $a_1\in A_1$ and consider the family of functions $$\mc{F}':=\{f(a_1,\cdot): f\in\mc{F}\}.$$ 
Observe that the functions from $\mc{F}'$ are defined on an $\mathscr{A}$-cross $\mathbf{Q}'({\alpha}'^i_j),$ where the defining matrix $({\alpha}'^i_j)_{j=1,\ldots,N-1,i=1,\ldots,l'_1+\ldots+l'_{N-1}}$ comes from the set of all $\alpha(s)^r_0$'s after ordering its elements with respect to the lexicographical order. The $\mathscr{A}$-cross $\mathbf{Q}'({\alpha}'^i_j),$ however, may not be reduced. Nevertheless, the cancelling all the branches of it, which are contained in another ones, does not change anything here, so we may assume without loss of generality that $\mathbf{Q}'({\alpha}'^i_j)$ is reduced. By the inductive assumption (observe that all assumptions of the theorem are now satisfied), for any function $g\in \mc{F}'$, the function $g\circ\sigma$ extends holomorphically to $\widehat{\mathbf{X}}'=\widehat{\mb{X}}'(A_{m_0'},B';D_{m_0'},G')$ with some domain of holomorphy $G'\s D_2\times\ldots\times D_{m_0'-1}\times D_{m_0'+1}\times\ldots\times D_N$ and locally pluriregular set $B'\s G',$
where 
$$\sigma: D_{m_0'}\times D_2\times\ldots\times D_{m_0'-1}\times D_{m_0'+1}\times\ldots\times D_N\rightarrow D_2\times\ldots\times D_N$$ 
sends a point $(z_{m_0'},z_2,\ldots,z_{m_0'-1},z_{m_0'+1},\ldots,z_N)$ to the point $(z_2,\ldots,z_N).$\\ 
\indent Furthermore, there exists a minimal number $n_0\in\{1,\ldots,l_1+\ldots+l_N\}$ such that for every $n'\geq n_0$ there is $\alpha^{n'}_1=1$ while for $n''<n_0$ we have $\alpha^{''}_1=0.$ Fix a point $z'\in\mathbf{Q}'':=\mathbf{Q}''(\alpha ''^i_j),$ where the matrix $(\alpha ''^i_j)_{j=1,\ldots,N-1,i=1,\ldots,l''_1+\ldots+l''_{N-1}}$ comes from the set $\{\alpha(s)^{n'}_0:n'\geq n,l_s\neq 0\}$ after ordering its elements with respect to the lexicographical order. Also, similarly as above, we may assume that $\mathbf{Q}''$ is reduced. Consider the family of functions $\mc{F}'':=\{f(\cdot,z'):f\in\mc{F}\}.$ Then every function from $\mc{F}''$ is holomorphic on $D_1.$\\
\indent Observe that $\mathbf{Q}''\s\sigma(\widehat{\mathbf{X}}'),$ since by the construction, for any row $\alpha ''^i$ of~the~matrix $(\alpha ''^i_j)$ there exists some row $\alpha '^k$ of the matrix $(\alpha '^i_j)$ which is subsequent to $\alpha ''^i$ with respect to the lexicographical order (i.e. any branch of~$\mathbf{Q}''$ is contained in some branch of $\mathbf{Q}'({\alpha}'^i_j)$). 
Now the conclusion holds true because of Theorem \ref{cross theorem} applied to the $2$-fold cross $(A_1\times\sigma(\widehat{\mathbf{X}}'))\cup(D_1\times\mathbf{Q}'')$ and $m_0=1,G=\sigma(\widehat{\mathbf{X}}'),B=\mathbf{Q}'',\tau=id$ are good for our purpose.
\end{proof}
\begin{rem}
Observe that in the above theorem (and its proof) the number $m_0$ might not be unique. However, since $\widehat{\mathbf{X}}$ is the envelope of holomorphy of $\mathbf{Q}$, the result must be the same, no matter which $m_0$ we distinct.\\
\indent To see this, observe that if $\widehat{\mathbf{X}}_1,\widehat{\mathbf{X}}_2$ are two envelopes of $\mathbf{Q}$ constructed for two different $m_0$'s, then for any function $f\in\mc{O}(\widehat{\mathbf{X}}_1)$ there exists a function $\widetilde{f}\in\mc{O}(\widehat{\mathbf{X}}_2)$ with $\widetilde{f}=f$ on the connected component of $\widehat{\mathbf{X}}_1\cap\widehat{\mathbf{X}}_1$ containing $\mathbf{Q}$ (notice that $\mathbf{Q}$ is connected - this follows from the proof of Theorem \ref{quasi} and main result from \cite{L1}). 
\end{rem}
\begin{cor}Let $s,N_1,\ldots, N_s\in\mb{N},s,N_1,\ldots, N_s\geq 2,$
let $D_{d_j}$ be a~Riemann domain of holomorphy over $\mb{C}^{n_{d_j}}$ and $A_{d_j}\s D_{d_j}$ be locally pluriregular, $j=1,\ldots,N_d,d=1,\ldots,s$. Let $\mathbf{Q}=\mathbf{Q}_1\times\ldots\times\mathbf{Q}_s,$ where $\mathbf{Q}_d$ is an $\mathscr{A}$-cross spread over $D_{d_1}\times\ldots\times D_{d_{N_d}},d=1,\ldots,s.$ Assume that $\mb{X}_{N_1+\ldots+N_s,1}((A_j,D_j)_{j=1}^{N_1+\ldots+N_s})\s\mathbf{Q}.$ Then the envelope of holomorphy $\widehat{\mathbf{Q}}$ of $\mathbf{Q}$ equals $\widehat{\mathbb{X}}_{s,1}((\mathbf{Q_j},\widehat{\mathbf{Q}}_j)_{j=1}^s).$
\end{cor}
\begin{proof}
Without loss of generality we may assume that $\mathbf{Q}_j$ is reduced, $j=1,\ldots,s.$\\
\indent Fix an $f\in\mc{F}:=\mc{O}_s(\mathbf{Q}).$ For any $j\in\{1,\ldots,s\}$ and for any point $$(z_1,\ldots,z_{j-1},z_{j+1},\ldots,z_s)\in\mathbf{Q}_1\times\ldots\times\mathbf{Q}_{j-1}\times\mathbf{Q}_{j+1}\times\ldots\times\mathbf{Q}_s$$
the function $f(z_1,\ldots,z_{j-1},\cdot,z_{j+1},\ldots,z_s)$ extends holomorphically to $\widehat{\mathbf{Q}}_j.$ Therefore, any function from $\mc{F}$ may be treated as a separately holomorphic function on ${\mathbb{X}}_{s,1}((\mathbf{Q_j},\widehat{\mathbf{Q}}_j)_{j=1}^s)$, which finishes the proof. 
\end{proof}
\begin{rem}\label{global}
Let $\mathbf{Q}$ be as in Theorem \ref{quasi} and let $\widehat{\mathbf{X}}$ be constructed via Theorem \ref{quasi}. We may consider the extension theorem for $\mathbf{Q}$ with analytic singularities given by the analytic subset $G$ of positive pure codimension of~$\widehat{\mathbf{X}}.$ Using the analogous argument to the one given in the proof of~Theorem 2.12 from \cite{L1} we can state and prove the extension theorem for $\mathscr{A}$-crosses with analytic singularities, parallel to Theorem 2.12 (case $F=\varnothing$) therein. 
\end{rem}
Note that the set $G$ in the above Remark is an analytic subset of $\widehat{\mathbf{X}}.$ The~question is, whether we can prove a ``local'' version of the extension theorem, when $G$ is  an analytic subset of~an~open neighborhood $U\s\widehat{\mathbf{X}}$ of $\mathbf{Q}$ (not necessarily $U=\widehat{\mathbf{X}})$. Concerning this matter, we have the following
\begin{theorem}\label{quasising}
Let $D_j$ be a Riemann domain of holomorphy over $\mb{C}^{n_j}$ and $A_j\s D_j$ be locally pluriregular, compact, and holomorphically convex, $j=1,\ldots, N.$ Put $\mathbf{Q}:=\mb{Q}(\alpha^i_j)((A_j,D_j)_{j=1}^N)$. Assume that $\mb{X}_{N,1}((A_j,D_j)_{j=1}^N)\s\mathbf{Q}$. Denote by $\widehat{\mathbf{Q}}$ the envelope of holomorphy of $\mathbf{Q}$ (cf. Theorem \ref{quasi}).
Define $M:=\mathbf{Q}\cap G,$ where $G$ is an analytic subset of an open neighborhood $U$ of $\mathbf{Q}$ contained in $\widehat{\mathbf{Q}}$ with {\normalfont\ codim}$G\geq 1$ and let $\mc{F}:=\mc{O}_s(\mathbf{Q}\setminus M).$ Then there exist an analytic set $\widehat{M}\s\widehat{\mathbf{Q}}$ and an open neighborhood $U_0\s U$ of $\mathbf{Q}$ 
such that:\\
\indent$\bullet$ $\widehat{M}\cap U_0\s G,$\\
\indent$\bullet$ for any $f\in\mc{F}$ there exists an $\widehat{f}\in\mc{O}(\widehat{\mathbf{Q}}\setminus\widehat{M})$ such that $\widehat{f}=f$ on $\mathbf{Q}\setminus M,$\\
\indent$\bullet$ $\widehat{M}$ is singular with respect to the family $\{\widehat{f}:f\in\mc{F}\}$ (see \cite{J1}, Chapter 3, Section 4),\\
\indent$\bullet$ $\widehat{f}(\widehat{\mathbf{Q}}\setminus\widehat{M})\s f(\mathbf{Q}\setminus M)$ for any $f\in\mc{F}.$\\
In particular, the extension result holds for the $(N,k)$-crosses.  
\end{theorem}
\begin{proof}It suffices to prove that every function $f\in\mc{F}$ admits a holomorphic extension $\widetilde{f}\in\mc{O}(U_0\setminus G),$ where $U_0$ is some open connected neighborhood of $\mathbf{Q},$ contained in $U.$\\
\indent Indeed, suppose for a moment that we proved the existence of $U_0$. Observe that then $\widehat{\mathbf{Q}}$ is the envelope of holomorphy of $U_0.$ To see this, let $g\in\mc{O}(U_0).$ Then $g|_{\mathbf{Q}}\in\mc{O}_s(\mathbf{Q}).$ In view of Theorem \ref{quasi}, there exists a $\widehat{g}\in\mc{O}(\widehat{\mathbf{Q}})$ such that $\widehat{g}=g$ on $\mathbf{Q},$ which implies that $\widehat{g}=g$ on $U_0$.\\ 
\indent Now it suffices to use the Dloussky theorem (see \cite{Dl}, see also \cite{Por}) and the proof is finished.\\
\indent For any $j\in\{1,\ldots,N\}$ take an exhausting sequence of~relatively compact domains of holomorphy in $D_j$ containing $A_j,$ say $\mc{F}_j=(F^j_n)_{n=1}^{\infty}.$ For any $n\in\mb{N}$ put $\mathbf{Q}_n:=\mb{Q}(\alpha^i_j)((A_j,F^j_n)_{j=1}^N).$ We shall construct the $U_0$ (in three steps).\\
\indent\textit{Step 1.} There exists a $U_c\s U,$ an open neighborhood of $A_1\times\ldots\times A_N$ (not necessarily connected) such that for every $f\in\mc{F}$ there exists an $\widehat{f}\in\mc{O}(U_c\setminus G)$ such that $f=\widehat{f}$ on $(U_c\cap\mathbf{Q})\setminus G.$\\
\indent From the existence of Stein neighborhood bases for the $A_j$'s, we find pseudoconvex sets $A_j\s V_j\s F^j_1, j=1,\ldots, N$ such that $V_1\times\ldots\times V_N\s U.$ We may also choose $V_j$'s so that every connected component of every $V_j$ intersects $A_j$.\\
\indent If now $\widetilde{V_j}$ is any connected component of $V_j$ for $j=1,\ldots,N,$ then we may apply the extension theorem with ``global" singularities for an $\ms{A}$-cross $\mathbf{Q}((\widetilde{V_j})_{j=1}^N):=\mb{Q}(\alpha^i_j)((A_j\cap\widetilde{V_j},\widetilde{V_j})_{j=1}^N)$ (cf. Remark \ref{global}).\\
\indent Summing up all (disjoint) envelopes $\widehat{\mathbf{Q}}((\widetilde{V_j})_{j=1}^N)$, and denoting the sum by $U_c$, we conclude that $A_1\times\ldots\times A_N\s U_c$ and for any $f\in\mc{F}$ there exists an $\widehat{f}\in\mc{O}(U_c\setminus G)$ such that $f=\widehat{f}$ on $(U_c\cap\mathbf{Q})\setminus G.$\\ 
\indent Indeed, the inclusion is apparent. Further, let $a=(a_1,\ldots,a_N)\in(\mc{X}_{\alpha}\cap U_c)\setminus G,$ where $\mc{X}_{\alpha}$ is some branch of $\mathbf{Q}.$ Then there is a unique system $(\widetilde{V}^a_j)_{j=1}^N$ of connected components of $V_j$'s with $a_j\in\widetilde{V}_j.$ Therefore $f(a)=\widehat{f}(a),$ where
$$
\widehat{f}(z):=f_{(\widetilde{V}^z_j)_{j=1}^N}(z),\quad z\in U_c\setminus G,
$$
and $f_{(\widetilde{V}^z_j)_{j=1}^N}\in\mc{O}(\widehat{\mathbf{Q}}((\widetilde{V}^z_j)_{j=1}^N)\setminus G)$ is the extension of the restriction of $f$ to $\mathbf{Q}((\widetilde{V}^z_j)_{j=1}^N)\setminus G.$\\
\indent\textit{Step 2.} For each $n\in\mb{N}$ let $A_j\s W_n^j\s V_j,j=1,\ldots,N$ be pseudoconvex sets such that: 
\begin{enumerate}[\upshape (1)]
\item any connected component of any $W_n^j$ contains some part of $A_j$,
\item $W_n^1\times\ldots\times W_n^N\s U_c$, 
\item for any branch of $\mathbf{Q}_n,$ denoted by $\mc{X}_{\alpha}^n$, the set 
$\widetilde{\mc{X}}^n_{\alpha,1}\times\ldots\times\widetilde{\mc{X}}^n_{\alpha,N}$ is contained in $U$, where
\begin{displaymath}
\widetilde{\mc{X}}^n_{\alpha,j}:=
\begin{cases}
\mc{X}^n_{\alpha,j}=F_n^j,\ &\ {\normalfont{\rm\ if\ }} \alpha_j=1,\\
W_n^j, \ &\ {\normalfont{\rm\ if\ }} \alpha_j=0
\end{cases}
\end{displaymath}
(this is possible since $\mc{X}^n_{\alpha}\s\mc{X}_{\alpha}\s U$ and $\mc{X}^n_{\alpha}\s\s U$, where $\mc{X}_{\alpha}$ is the suitable branch of $\mathbf{Q}$). 
\end{enumerate}
We may also choose $((W_n^j)_{n=1}^{\infty})_{j=1}^N$ so that $W^j_{n+1}\s W_n^j,n\in\mb{N},j=1,\ldots,N.$\\
\indent Fix an $n\in\mb{N}.$ Let $\mc{X}^n_{\alpha}$ be any branch of $\mathbf{Q}_n.$ We shall construct an open (not necessarily connected) neighborhood $U^n_{\alpha}\s U$ of $\mc{X}^n_{\alpha}$ such that for every $f\in\mc{F}$ there exists an $f_{\alpha}^n\in\mc{O}(U^n_{\alpha}\setminus G)$ such that $f=f_{\alpha}^n$ on $(\mathbf{Q}_n\cap U_{\alpha}^n)\setminus G$.\\ 
\indent Without loss of generality we may assume that $\alpha=(\underbrace{0,\ldots,0}_{N-s},\underbrace{1,\ldots,1}_{s})$. Fix an $f\in\mc{F}.$ We know that for any $a\in A_{\alpha}=A_1\times\ldots\times A_{N-s}$, the function $f\circ\boldsymbol{i}_{a,\alpha}=f(a,\cdot)$ is holomorphic on the set $D_{\alpha}\setminus G_a=(D_{N-s+1}\times\ldots\times D_N)\setminus G_a.$ On the other hand, using Step 1, we see that for a fixed $b=(b_{N-s+1},\ldots,b_N)\in A_{N-s+1}\times\ldots\times A_N,$ the function $f(\cdot,b)$ extends to a function $$f_b\in\mc{O}((\widetilde{W_{n+1}^1}\times\ldots\times\widetilde{W^{N-s}_{n+1}})\setminus G_b),$$ 
where for each $j$, $\widetilde{W^j_{n+1}}$ denotes some connected component of $W_{n+1}^j$.\\
\indent Indeed, it is enough to observe that $f(\cdot,b)=\widehat{f}(\cdot,b)$ on the set $$S:=(\mathbf{Q}_b\cap(\widetilde{W_{n+1}^1}\times\ldots\times\widetilde{W^{N-s}_{n+1}}))\setminus G_b.$$ 
To do this, let $x\in S.$ Then $y:=(x,b)\in\mathbf{Q}\setminus G$ and therefore $f(y)=\widehat{f}(y).$\\ 
\indent Consider the $2$-fold cross 
\begin{multline*}
\mathbf{Z}:=((\widetilde{W_{n+1}^1}\times\ldots\times\widetilde{W^{N-s}_{n+1}})\times (A_{N-s+1}\times\ldots\times A_N))\\
\cup(((A_1\cap\widetilde{W^1_{n+1}})\times\ldots\times(A_{N-s}\cap\widetilde{W^{N-s}_{n+1}}))\times F_{n+1}^{N-s+1}\times\ldots\times F^N_{n+1}).
\end{multline*} 
Let $F:\mathbf{Z}\setminus G\rightarrow\mb{C}$ be defined as follows
\begin{displaymath}
F(z,w):=\begin{cases}
f(z,w), & {\rm\ if\ } (z,w)\in\left(\Big(\prod\limits_{j=1}^{N-s}(A_j\cap\widetilde{W^j_{n+1}})\Big)\times\Big(\prod\limits_{k=N-s+1}^NF_{n+1}^k\Big)\right)\setminus G,\\
f_w(z), & {\rm\ if\ } (z,w)\in\left(\Big(\prod\limits_{j=1}^{N-s}\widetilde{W^j_{n+1}}\Big)\times\Big(\prod\limits_{k=N-s+1}^NA_k\Big)\right)\setminus G.
\end{cases}
\end{displaymath}
Observe that $F$ is well-defined and separately holomorphic on $\mathbf{Z}\setminus G.$ Therefore, making use of the extension theorem with ``global" analytic singularities, we conclude that $F$ extends to an $\widetilde{F}\in\mc{O}(\widehat{\mathbf{Z}}\setminus G).$ Now we may glue together all the extensions related to all the systems of connected components of $W^j_{n+1}$'s (observe that all the extension domains are pairwise disjoint) to get a function $\widehat{F}\in\mc{O}(\Omega\setminus G),$ where 
\begin{multline*}
A_1\times\ldots\times A_{N-s}\times F^{N-s+1}_{n+1}\times\ldots\times F_{n+1}^N\s\Omega\s\\
 W^1_{n+1}\times\ldots\times W^{N-s}_{n+1}\times F^{N-s+1}_{n+1}\times\ldots\times F^N_{n+1}\s U
\end{multline*} 
is an open neighborhood of $\mc{X}^n_{\alpha}.$\\ 
\indent Using the fact that $F^j_n\s\s F_{n+1}^j,j=1,\ldots,N,n\in\mb{N}$,
we choose 
$$A_1\s\Gamma^1_n\s W^1_{n+1},\ldots,A_{N-s}\s\Gamma^{N-s}_n\s W^{N-s}_{n+1},$$ open neighborhoods (and yet not necessarily pseudoconvex) with the property that every connected component of every $\Gamma^j_n$ intersects $A_j$, and satisfying 
$$U_{\alpha}^n:=\Gamma_n^1\times\ldots\times\Gamma^{N-s}_n\times F^{N-s+1}_n\times\ldots\times F^{N}_n\s\Omega.$$ 
\indent Consider the restriction of $\widehat{F}$ to $\mc{D}_{\widehat{F}}:=U_{\alpha}^n\setminus G.$ This truncated (holomorphic) function will be denoted by $f_{\alpha}^n$. Observe that in this situation $f=f_{\alpha}^n$ on $\mc{D}_{\widehat{F}}\cap\mathbf{Q}_n.$\\
\indent To see this, let $a\in\mc{D}_{\widehat{F}}\cap\mc{X}_{\beta}^n,$ where $\mc{X}_{\beta}^n$ is any branch of $\mathbf{Q}_n.$ Then 
$$
f\circ\boldsymbol{i}_{a^0_{\beta},\beta}\in\mc{O}(D_{\beta}\setminus G_{a^0_{\beta}}),
$$ 
as well as
$$
f_{\alpha}^n\circ\boldsymbol{i}_{a^0_{\beta},\beta}\in\mc{O}(E_{\beta}),
$$
where $E_{\beta}\s D_{\beta}\setminus G_{a^0_{\beta}}$ and it is such that any of its connected components contains a nonpluripolar set on which $f\circ\boldsymbol{i}_{a^0_{\beta},\beta}$ and $f_{\alpha}^n\circ\boldsymbol{i}_{a^0_{\beta},\beta}$ coincide. It is then enough to use the identity principle.\\
\indent\textit{Step 3.} Observe that for any $n\in\mb{N}$, any $j=1,\ldots, N$, and for any $\alpha$ from the defining matrix of $\mathbf{Q}$ such that $\alpha_j=0$ we have choosen some $\Gamma^j_n(\alpha).$ It is apparent that for each $n\in\mb{N}$ and each $j=1,\ldots,N$ we have $A_j\s\bigcap\limits_{\alpha:\alpha_j=0}\Gamma^j_n(\alpha).$ Therefore, shrinking $\Gamma_n^j(\alpha)$'s we may assume that for every $j=1,\ldots,N$, every $n\in\mb{N}$, and for any couple $\alpha^1,\alpha^2$ with $\alpha_j^k=0,k=1,2$ we have $\Gamma^j_n(\alpha^1)=\Gamma^j_n(\alpha^2)$ (and all other properties are kept).\\
\indent Moreover, we may assume that for any triple $(n,j,\alpha),$ where $n\in\mb{N}$, $j\in\{1,\ldots,N\}$, and $\alpha$ is from the defining matrix of $\mathbf{Q}$ and such that $\alpha_j=0$, we have 
$$
\Gamma^j_{n+1}(\alpha)\s \Gamma^j_n(\alpha).
$$
Indeed, we choose $\Gamma_1^j$'s as in Step 2. Then assuming we have chosen $\Gamma^j_1,\ldots,\Gamma^j_{n}$, $j=1,\ldots,N,$ we take $\Gamma^{j}_{n+1}$ as in Step 2 and so that $\Gamma_{n+1}^j\s(W^j_{n+2}\cap\Gamma^j_n),j=1,\ldots,N.$\\
\indent Observe that for any couple of branches of $\mathbf{Q},$ say $\mc{X}_{\alpha},\mc{X}_{\beta}$, and for any $(n,k)\in\mb{N}^2$, if $(U_{\alpha}^n\cap U^k_{\beta})\setminus G\neq\varnothing$, then $f_{\alpha}^n=f_{\beta}^k$ on $(U_{\alpha}^n\cap U_{\beta}^k)\setminus G.$ This follows directly from the construction of $U_{\alpha}^n$ and $U^k_{\beta}.$\\
\indent Therefore, we may glue all the $U^n_{\alpha}$'s and $f^n_{\alpha}$'s together and we finish the proof by taking $U_0$ to be the connected component of $\bigcup\limits_{\alpha}\bigcup\limits_nU^n_{\alpha}$ containing $\mathbf{Q}$.    
\end{proof}
\section{``Nice'' descriptions and some geometry}
\label{sec4}
We know that in the context of $(N,k)$-crosses and generalized $(N,k)$-crosses, their envelopes of holomorphy have a nice description in terms of the relative extremal function of the set $A_j$ with respect to $D_j$. For the $\mathscr{A}$-crosses the existence of such description is a more subtle problem. For $N\in\{2,3\}$ the situation is simple, as in this case the $\mathscr{A}$-crosses which are interesting from the point of view of Theorem \ref{quasi} (i.e. those which contain $\mb{X}_{2,1}((A_j,D_j)_{j=1,2})$, and $\mb{X}_{3,1}((A_j,D_j)_{j=1,2,3}),$ respectively) give nothing new in comparison with ``old'' crosses (see Example \ref{2,3}). However, already in the case $N=4$ such descriptions can be various.
\begin{ex}\label{nine}
Let $N=4.$ Let $D_j\s\mb{C}^{n_j}$ be a hyperconvex domain and let $A_j\s D_j$ be compact, locally pluriregular, and locally $L$-regular (see \cite{Si1}), $j=1,\ldots,4$. Consider the $\mathscr{A}$-crosses interesting from the point of view of Theorem \ref{quasi}. A straightforward computation shows that we have to consider exactly nine nontrivial (that is, different from the $(4,k)-$crosses) cases here (up to permutations of variables):\\ 
\begin{center}
\begin{displaymath}
\mathbf{Q}^1:=\mathbf{Q}\scriptsize\left(\!\!\begin{array}{cccc}0&0&0&1\\0&1&1&0\\1&0&0&0\end{array}\!\!\right)=
{\begin{array}{l}
(A_1\times A_2\times A_3\times D_4)\cup\\
(A_1\times D_2\times D_3\times A_4)\cup\\
(D_1\times A_2\times A_3\times A_4),
\end{array}}
\end{displaymath}
\end{center}
\begin{center}
\begin{displaymath}
\mathbf{Q}^2:=\mathbf{Q}\scriptsize\left(\!\!\begin{array}{cccc}0&0&0&1\\1&0&1&0\\1&1&0&0\end{array}\!\!\right)=
\begin{array}{l}
(A_1\times A_2\times A_3\times D_4)\cup\\
(D_1\times A_2\times D_3\times A_4)\cup\\
(D_1\times D_2\times A_3\times A_4),
\end{array}
\end{displaymath}
\end{center}
\begin{center}
\begin{displaymath}
\mathbf{Q}^3:=\mathbf{Q}\scriptsize\left(\!\!\begin{array}{cccc}0&0&1&1\\0&1&0&1\\0&1&1&0\\1&0&0&1\\1&0&1&0\end{array}\!\!\right)=
\begin{array}{l}
(A_1\times A_2\times D_3\times D_4)\cup\\
(A_1\times D_2\times A_3\times D_4)\cup\\
(A_1\times D_2\times D_3\times A_4)\cup\\
(D_1\times A_2\times A_3\times D_4)\cup\\
(D_1\times A_2\times D_3\times A_4),
\end{array}
\end{displaymath}
\end{center}
\begin{center}
\begin{displaymath}
\mathbf{Q}^4:=\mathbf{Q}\scriptsize\left(\!\!\begin{array}{cccc}0&0&1&1\\0&1&1&0\\1&0&0&1\\1&1&0&0\end{array}\!\!\right)
=
\begin{array}{l}
(A_1\times A_2\times D_3\times D_4)\cup\\
(A_1\times D_2\times D_3\times A_4)\cup\\
(D_1\times A_2\times A_3\times D_4)\cup\\
(D_1\times D_2\times A_3\times A_4),
\end{array}
\end{displaymath}
\end{center}
\begin{center}
\begin{displaymath}
\mathbf{Q}^5:=\mathbf{Q}\scriptsize\left(\!\!\begin{array}{cccc}0&1&1&1\\1&0&0&1\\1&1&0&0\end{array}\!\!\right)=
\begin{array}{l}
(A_1\times D_2\times D_3\times D_4)\cup\\
(D_1\times A_2\times A_3\times D_4)\cup\\
(D_1\times D_2\times A_3\times A_4),
\end{array}
\end{displaymath}
\end{center}
\begin{center}
\begin{displaymath}
\mathbf{Q}^6:=\mathbf{Q}\scriptsize\left(\!\!\begin{array}{cccc}0&0&1&1\\1&0&0&1\\1&1&0&0\end{array}\!\!\right)=
\begin{array}{l}
(A_1\times A_2\times D_3\times D_4)\cup\\
(D_1\times A_2\times A_3\times D_4)\cup\\
(D_1\times D_2\times A_3\times A_4),
\end{array}
\end{displaymath}
\end{center}
\begin{center}
\begin{displaymath}
\mathbf{Q}^7:=\mathbf{Q}\scriptsize\left(\!\!\begin{array}{cccc}0&0&1&1\\0&1&0&1\\1&0&0&1\\1&0&1&0\end{array}\!\!\right)=
\begin{array}{l}
(A_1\times A_2\times D_3\times D_4)\cup\\
(A_1\times D_2\times A_3\times D_4)\cup\\
(D_1\times A_2\times A_3\times D_4)\cup\\
(D_1\times A_2\times D_3\times A_4),
\end{array}
\end{displaymath}
\end{center}
\begin{center}
\begin{displaymath}
\mathbf{Q}^8:=\mathbf{Q}\scriptsize\left(\!\!\begin{array}{cccc}0&0&0&1\\0&1&1&0\\1&0&1&0\\1&1&0&0\end{array}\!\!\right)=
\begin{array}{l}
(A_1\times A_2\times A_3\times D_4)\cup\\
(A_1\times D_2\times D_3\times A_4)\cup\\
(D_1\times A_2\times D_3\times A_4)\cup\\
(D_1\times D_2\times A_3\times A_4),
\end{array}
\end{displaymath}
\end{center}
\begin{center}
\begin{displaymath}
\mathbf{Q}^9:=\mathbf{Q}\scriptsize\left(\!\!\begin{array}{cccc}0&1&1&1\\1&0&0&1\\1&0&1&0\\1&1&0&0\end{array}\!\!\right)=
\begin{array}{l}
(A_1\times D_2\times D_3\times D_4)\cup\\
(D_1\times A_2\times A_3\times D_4)\cup\\
(D_1\times A_2\times D_3\times A_4)\cup\\
(D_1\times D_2\times A_3\times A_4).
\end{array}
\end{displaymath}
\end{center}
Then, after some simple computation we get:
\begin{multline*}
\widehat{\mathbf{Q}}^1=\{(z_1,z_2,z_3,z_4)\in D_1\times D_2\times D_3\times D_4:\boldsymbol{h}^{\star}_{A_1,D_1}(z_1)+\boldsymbol{h}^{\star}_{A_4,D_4}(z_4)+\\\max\{\boldsymbol{h}^{\star}_{A_2,D_2}(z_2),\boldsymbol{h}^{\star}_{A_3,D_3}(z_3)\}<1\},
\end{multline*}
\begin{multline*}
\widehat{\mathbf{Q}}^2=\{(z_1,z_2,z_3,z_4)\in D_1\times D_2\times D_3\times D_4:\boldsymbol{h}^{\star}_{A_4,D_4}(z_4)+\\\max\{\boldsymbol{h}^{\star}_{A_1,D_1}(z_1),\boldsymbol{h}^{\star}_{A_2,D_2}(z_2)+\boldsymbol{h}^{\star}_{A_3,D_3}(z_3)\}<1\},
\end{multline*}
\begin{multline*}
\widehat{\mathbf{Q}}^3=\{(z_1,z_2,z_3,z_4)\in D_1\times D_2\times D_3\times D_4:\boldsymbol{h}^{\star}_{A_1,D_1}(z_1)+\boldsymbol{h}^{\star}_{A_2,D_2}(z_2)+\\\max\{\boldsymbol{h}^{\star}_{A_3,D_3}(z_3)+\boldsymbol{h}^{\star}_{A_4,D_4}(z_4)-1,0\}<1\},
\end{multline*}
\begin{multline*}
\widehat{\mathbf{Q}}^4=\{(z_1,z_2,z_3,z_4)\in D_1\times D_2\times D_3\times D_4:\\ \max\{\boldsymbol{h}^{\star}_{A_2,D_2}(z_2)+\boldsymbol{h}^{\star}_{A_4,D_4}(z_4),\boldsymbol{h}^{\star}_{A_1,D_1}(z_1)+\boldsymbol{h}^{\star}_{A_3,D_3}(z_3)\}<1\},
\end{multline*}
\begin{multline*}
\widehat{\mathbf{Q}}^5=\{(z_1,z_2,z_3,z_4)\in D_1\times D_2\times D_3\times D_4:\boldsymbol{h}^{\star}_{A_1,D_1}(z_1)+\\ \max\{\boldsymbol{h}^{\star}_{A_3,D_3}(z_3),\max\{\boldsymbol{h}^{\star}_{A_2,D_2}(z_2)+\boldsymbol{h}^{\star}_{A_4,D_4}(z_4)-1,0\}\}<1\}.
\end{multline*}
Observe that until now we only used the fact that $D_j$'s are pseudoconvex and $A_j$'s are locally pluriregular.\\
\indent The remaining cases are a little bit more complicated.\\
\indent Consider $\mathbf{Q}^6$ and take a function $f\in\mc{F}.$ Observe that:\\
$\bullet$ For any fixed point $a_2\in A_2$ the function $f(\cdot,a_2,\cdot)$ is holomorphic on the set $(\widehat{(A_1\times D_3)\cup(D_1\times A_3)})\times D_4.$ Moreover, for any point $(z_1,z_3,z_4)\in D_1\times A_3\times A_4,$ the function $f(z_1,\cdot,z_3,z_4)$ is holomorphic on $D_2$. Therefore, using Theorem \ref{cross theorem}, we conclude that the envelope of holomorphy of $\mathbf{Q}^6$ equals (up to permutation of variables)
\begin{multline*}
\{(z_2,z_1,z_3,z_4)\in D_2\times(\widehat{(A_1\times D_3)\cup(D_1\times A_3)})\times D_4:\boldsymbol{h}^{\star}_{A_2,D_2}(z_2)+\\\boldsymbol{h}^{\star}_{D_1\times A_3\times A_4,(\widehat{(A_1\times D_3)\cup(D_1\times A_3)})\times D_4}(z_1,z_3,z_4)<1\}.
\end{multline*}
$\bullet$ On the other hand, for any point $(z_1,z_3)\in D_1\times A_3$, the function $f(z_1,\cdot,z_3,\cdot)$ is holomorphic on $\widehat{(A_2\times D_4)\cup(D_2\times A_4)}.$ Moreover, for any point $(z_2,z_4)\in A_2\times D_4$,
the function $f(\cdot,z_2,\cdot,z_4)$ is holomorphic on $\widehat{(A_1\times D_3)\cup(D_1\times A_3)}.$ Using once again Theorem \ref{cross theorem}, we see that the envelope of holomorphy of $\mathbf{Q}$ equals (up to permutation of variables)
\begin{multline*}
\{(z_2,z_4,z_1,z_3)\in \widehat{(A_2\times D_4)\cup(D_2\times A_4)}\times\widehat{(A_1\times D_3)\cup(D_1\times A_3)}:\\
\boldsymbol{h}^{\star}_{D_1\times A_3,\widehat{(A_1\times D_3)\cup(D_1\times A_3)}}(z_1,z_3)+\boldsymbol{h}^{\star}_{A_2\times D_4,\widehat{(A_2\times D_4)\cup(D_2\times A_4)}}(z_2,z_4)<1
\}.
\end{multline*}
\indent From the above bullets it follows that any arbitrary point $(z_1,z_2,z_3,z_4)\in D_1\times D_2\times D_3\times D_4$ satisfies the following system of conditions:
\begin{align}\label{A}
\boldsymbol{h}^{\star}_{A_2,D_2}(z_2)+\boldsymbol{h}^{\star}_{A_4,D_4}(z_4)<1,
\end{align}
\begin{align}\label{B}
\boldsymbol{h}^{\star}_{A_1,D_1}(z_1)+\boldsymbol{h}^{\star}_{A_3,D_3}(z_3)<1,
\end{align}
\begin{align}\label{C}
\boldsymbol{h}^{\star}_{D_1\times A_3,\widehat{(A_1\times D_3)\cup(D_1\times A_3)}}(z_1,z_3)+\boldsymbol{h}^{\star}_{A_2\times D_4,\widehat{(A_2\times D_4)\cup(D_2\times A_4)}}(z_2,z_4)<1
\end{align}
iff it satisfies the following system of conditions:
\begin{align}\label{D}
\boldsymbol{h}^{\star}_{A_2,D_2}(z_2)+\boldsymbol{h}^{\star}_{A_4,D_4}(z_4)<1,
\end{align}
\begin{align}\label{E}
\boldsymbol{h}^{\star}_{A_1,D_1}(z_1)+\boldsymbol{h}^{\star}_{A_3,D_3}(z_3)<1,
\end{align}
\begin{align}\label{F}
\boldsymbol{h}^{\star}_{A_2,D_2}(z_2)+\boldsymbol{h}^{\star}_{D_1\times A_3,\widehat{(A_1\times D_3)\cup(D_1\times A_3)}}<1.
\end{align}
We shall prove the following\\
\emph{Claim.} 
\begin{displaymath}
\boldsymbol{h}^{\star}_{A_2\times D_4,\widehat{(A_2\times D_4)\cup(D_2\times A_4)}}(z_2,z_4)=\boldsymbol{h}^{\star}_{A_2,D_2}(z_2)
\end{displaymath} 
for $(z_2,z_4)\in\widehat{(A_2\times D_4)\cup(D_2\times A_4)}.$
\begin{proof}[Proof of Claim] Observe that the inequality $\geq$ is evident. 
Suppose, seeking a contradiction, that there exists a point $(z^0_2,z^0_4)\in\widehat{(A_2\times D_4)\cup(D_2\times A_4)}$, and numbers $\alpha,\beta\in(0,1)$ such that
\begin{displaymath}
\alpha=\boldsymbol{h}^{\star}_{A_2\times D_4,\widehat{(A_2\times D_4)\cup(D_2\times A_4)}}(z^0_2,z^0_4)>\boldsymbol{h}^{\star}_{A_2,D_2}(z^0_2)=\alpha-\beta.
\end{displaymath} 
We know (Proposition 4.5.2 from \cite{K1}) that there exists a $z^0_3\in D_3$ such that $\boldsymbol{h}^{\star}_{A_3,D_3}(z^0_3)=1-\alpha.$\\
Finally, take any $z^0_1\in A_1.$ Then
\begin{multline*}
1-\alpha=\boldsymbol{h}^{\star}_{A_3,D_3}(z^0_3)\leq \boldsymbol{h}^{\star}_{D_1\times A_3,\widehat{(D_1\times A_3)\cup(A_1\times D_3)}}(z^0_1,z^0_3)\leq\\
\boldsymbol{h}^{\star}_{A_1\times A_3,\widehat{(D_1\times A_3)\cup(A_1\times D_3)}}(z^0_1,z^0_3)=\boldsymbol{h}^{\star}_{A_1,D_1}(z^0_1)+\boldsymbol{h}^{\star}_{A_3,D_3}(z^0_3)=1-\alpha.
\end{multline*}
Therefore, $\boldsymbol{h}^{\star}_{D_1\times A_3,\widehat{(D_1\times A_3)\cup(A_1\times D_3)}}(z^0_1,z^0_3)=1-\alpha.$ Observe that in this situation the point $(z^0_1,z^0_2,z^0_3,z^0_4)$ satisfies the conditions (\ref{A}), (\ref{B}), (\ref{D}), (\ref{E}), and (\ref{F}), while it does not satisfy the condition (\ref{C}), which is a~contradiction.
\end{proof}
Making use of the above claim we conclude that the description under our interest is:
\begin{multline*}
\widehat{\mathbf{Q}}^6=\{(z_1,z_2,z_3,z_4)\in D_1\times D_2\times D_3\times D_4:\max\{\boldsymbol{h}^{\star}_{A_1,D_1}(z_1)+\boldsymbol{h}^{\star}_{A_3,D_3}(z_3),\\
\boldsymbol{h}^{\star}_{A_2,D_2}(z_2)+\boldsymbol{h}^{\star}_{A_4,D_4}(z_4),\boldsymbol{h}^{\star}_{A_2,D_2}(z_2)+\boldsymbol{h}^{\star}_{A_3,D_3}(z_3)\}<1\}.
\end{multline*}
Note here we needed the hyperconvexity of $D_j$'s. Also, we used only the local pluriregularity of $A_j$'s. The full system of assumptions for $A_j$'s (with $D_j$'s pseudoconvex) is needed to solve the last three cases.\\
\indent Consider $\mathbf{Q}^7$ and take a function $f\in\mc{F}.$ Observe that:\\
$\bullet$ For any fixed point $a_2\in A_2$, the function $f(\cdot,a_2,\cdot)$ is holomorphic on $\widehat{\mb{X}}_{3,2}((A_j,D_j)_{j=1,3,4}).$ Moreover, for any point $(z_1,z_3,z_4)\in A_1\times A_3\times D_4,$ the function $f(z_1,\cdot,z_3,z_4)$ is holomorphic on $D_2$. Therefore, using Theorem \ref{cross theorem}, we conclude that the envelope of holomorphy of $\mathbf{Q}^7$ equals (up to permutation of variables)
\begin{multline*}
\{(z_2,z_1,z_3,z_4)\in D_2\times \widehat{\mb{X}}_{3,2}((A_j,D_j)_{j=1,3,4}):\\
\boldsymbol{h}^{\star}_{A_2,D_2}(z_2)+\boldsymbol{h}^{\star}_{A_1\times A_3\times D_4,\widehat{\mb{X}}_{3,2}((A_j,D_j)_{j=1,3,4})}(z_1,z_3,z_4)<1
\}.
\end{multline*}
We shall prove the following\\
\emph{Claim.} 
\begin{multline*}
\boldsymbol{h}^{\star}_{A_1\times A_3\times D_4,\widehat{\mb{X}}_{3,2}((A_j,D_j)_{j=1,3,4})}(z_1,z_3,z_4)=\\
\max\{\boldsymbol{h}^{\star}_{A_1,D_1}(z_1),\boldsymbol{h}^{\star}_{A_3,D_3}(z_3),\boldsymbol{h}^{\star}_{A_1,D_1}(z_1)+\boldsymbol{h}^{\star}_{A_3,D_3}(z_3)+\boldsymbol{h}^{\star}_{A_4,D_4}(z_4)-1\}
\end{multline*}
for $(z_1,z_3,z_4)\in\widehat{\mb{X}}_{3,2}((A_j,D_j)_{j=1,3,4}).$
\begin{proof}[Proof of Claim] The inequality $\geq$ is obvious as the right-hand side is from the defining family for the left-hand side. In order to prove the equality, we may assume that $D_j$ is relatively compact and strongly pseudoconvex, $j=1,\ldots,4$ (use Proposition 3.2.25 from \cite{J2}).\\
\indent Choose an increasing sequence $(K_j)_{j=1}^{\infty}$ of holomorphically convex locally $L$-regular compacta in $D_4$ containing $A_4$ and such that $\bigcup_{j=1}^{\infty}K_j=D_4$ (this is possible by the existence of an exhausting sequence of holomorphically convex compacta in $D_4$ and using \cite{Ze2}; see also \cite{L2}). Put
$$
L_s(z_1,z_3,z_4):= \boldsymbol{h}^{\star}_{A_1\times A_3\times K_s,\widehat{\mb{X}}_{3,2}((A_j,D_j)_{j=1,3,4})}(z_1,z_3,z_4)
$$ 
for $s\in\mb{N}$ and $(z_1,z_3,z_4)\in\widehat{\mb{X}}_{3,2}((A_j,D_j)_{j=1,3,4}).$ Observe that the functions $L_s$ are continuous (use \cite{Si1}). Moreover, we have the equality $(dd^cL_s)^n=0$ on the set
$\widehat{\mb{X}}_{3,2}((A_j,D_j)_{j=1,3,4})\setminus~(A_1\times A_3\times K_s)=:V_s,s\in\mb{N},$ where $(dd^c)^n$ is the complex Monge-Amp{\`e}re operator (\cite{BT1}).\\
\indent Furthermore, for any $z_0\in\partial{V_s}$ there is
\begin{multline*}
\liminf_{V_s\ni z\rightarrow z_0}(\max\{\boldsymbol{h}^{\star}_{A_1,D_1}(z_1),\boldsymbol{h}^{\star}_{A_3,D_3}(z_3),\\\boldsymbol{h}^{\star}_{A_1,D_1}(z_1)+\boldsymbol{h}^{\star}_{A_3,D_3}(z_3)+\boldsymbol{h}^{\star}_{A_4,D_4}(z_4)-1\}
-L_s(z_1,z_3,z_4))\geq 0,
\end{multline*}
and, by the domination principle (Corollary 3.7.4 from \cite{K1}), 
\begin{multline*}
\boldsymbol{h}^{\star}_{A_1\times A_3\times D_4,\widehat{\mb{X}}_{3,2}((A_j,D_j)_{j=1,3,4})}(z_1,z_3,z_4)\leq L_s(z_1,z_3,z_4)=\\\max\{\boldsymbol{h}^{\star}_{A_1,D_1}(z_1),\boldsymbol{h}^{\star}_{A_3,D_3}(z_3),\boldsymbol{h}^{\star}_{A_1,D_1}(z_1)+\boldsymbol{h}^{\star}_{A_3,D_3}(z_3)+\boldsymbol{h}^{\star}_{A_4,D_4}(z_4)-1\}
\end{multline*}
for $s\in\mb{N}$ and $(z_1,z_3,z_4)\in\widehat{\mb{X}}_{3,2}((A_j,D_j)_{j=1,3,4}),$ 
from which follows the conslusion.
\end{proof}
Making use of the above claim we conclude that our description is given by
\begin{multline*}
\widehat{\mathbf{Q}}^7=\{(z_1,z_2,z_3,z_4)\in D_1\times D_2\times D_3\times D_4:
\boldsymbol{h}^{\star}_{A_2,D_2}(z_2)+\\\max\{\boldsymbol{h}^{\star}_{A_1,D_1}(z_1),\boldsymbol{h}^{\star}_{A_3,D_3}(z_3),\boldsymbol{h}^{\star}_{A_1,D_1}(z_1)+\boldsymbol{h}^{\star}_{A_3,D_3}(z_3)+\boldsymbol{h}^{\star}_{A_4,D_4}(z_4)-1\}<1
\}.
\end{multline*}
\indent The envelope of holomorphy of $\mathbf{Q}^8$ is of the form
\begin{multline*}
\{(z_1,z_2,z_3,z_4)\in \widehat{\mb{X}}_{3,2}((A_j,D_j)_{j=1}^3)\times D_4:\boldsymbol{h}^{\star}_{A_4,D_4}(z_4)+\\\boldsymbol{h}^{\star}_{A_1\times A_2\times A_3,\widehat{\mb{X}}_{3,2}((A_j,D_j)_{j=1}^3)}<1\}.
\end{multline*}
Using Proposition \ref{center} below we see that here the desired description is given by
\begin{multline*}
\widehat{\mathbf{Q}}^8=\Big\{(z_1,z_2,z_3,z_4)\in D_1\times D_2\times D_3\times D_4:
\boldsymbol{h}^{\star}_{A_4,D_4}(z_4)+\\
\max\Big\{\frac{1}{2}(\boldsymbol{h}^{\star}_{A_1,D_1}(z_1)+\boldsymbol{h}^{\star}_{A_2,D_2}(z_2)+\boldsymbol{h}^{\star}_{A_3,D_3}(z_3)),\max_{j=1,2,3}\{\boldsymbol{h}^{\star}_{A_j,D_j}(z_j)\}\Big\}<1\Big\}.
\end{multline*}
\indent As far as $\mathbf{Q}^9$ is concerned, its envelope of holomorphy is of the form
\begin{multline*}
\{(z_1,z_2,z_3,z_4)\in D_1\times D_2\times D_3\times D_4:\boldsymbol{h}^{\star}_{A_1,D_1}(z_1)+\\\boldsymbol{h}^{\star}_{\mb{X}_{3,1}((A_j,D_j)_{j=1}^3),D_1\times D_2\times D_3,}<1\}.
\end{multline*}
By Proposition \ref{envelope in envelope} below we get the following description:
\begin{multline*}
\widehat{\mathbf{Q}}^9=\{(z_1,z_2,z_3,z_4)\in D_1\times D_2\times D_3\times D_4:\boldsymbol{h}^{\star}_{A_1,D_1}(z_1)+\\
\frac{1}{2}(\boldsymbol{h}^{\star}_{A_2,D_2}(z_2)+\boldsymbol{h}^{\star}_{A_3,D_3}(z_3)+\boldsymbol{h}^{\star}_{A_4,D_4}(z_4)-1)<1\}
\end{multline*}
(cf. Example \ref{starting}).
\end{ex}
After the above example we could possibly expect that any set which can be described in a similar way as all envelopes of holomorphy from the above example, must be an envelope of holomorphy of certain $\mathscr{A}$-cross. This is however not the case. For instance, let $N=4$ and keep the assumptions from Example \ref{nine}. Put
\begin{multline*}
\tilde{\mathbf{Q}}:=\{(z_1,z_2,z_3,z_4)\in D_1\times D_2\times D_3\times D_4:\max\{\boldsymbol{h}^{\star}_{A_2,D_2}(z_2)+\boldsymbol{h}^{\star}_{A_4,D_4}(z_4),\\
\boldsymbol{h}^{\star}_{A_1,D_1}(z_1)+\boldsymbol{h}^{\star}_{A_2,D_2}(z_2)+\boldsymbol{h}^{\star}_{A_3,D_3}(z_3),\boldsymbol{h}^{\star}_{A_1,D_1}(z_1)+\boldsymbol{h}^{\star}_{A_3,D_3}(z_3)+\boldsymbol{h}^{\star}_{A_4,D_4}(z_4)\}<1\}.
\end{multline*}
It follows from some simple but a little bit tedious calculations that there is no $\mathscr{A}$-cross $\mathbf{Q}$ with the defining matrix of dimension $(l_1+l_2+l_3+l_4)\times 4$ such that the set $\tilde{\mathbf{Q}}$ is the envelope of holomorphy of $\mathbf{Q}.$ This gives rise to the (open) question whether there exists some other cross-like object $\mathbf{C}$ such that $\tilde{Q}$ is the envelope of holomorphy of $\mathbf{C}$.\\
\section{Appendix}
We prove two formulas for the relative extremal function, which were used in Example \ref{nine}, and which are interesting on their own.
\begin{prop}[cf. \cite{Si1}, \cite{J2}, Proposition 3.2.28]
\label{center} 
Let $D_j\s\mb{C}^{n_j}$ be a pseudoconvex domain and let $A_j\s D_j$ be locally pluriregular, compact and locally $L$-regular, $j=1,\ldots,N.$ Put $\mathbf{X}_{N,k}:=\mb{X}_{N,k}((A_j,D_j)_{j=1}^N).$ Then
\begin{multline*}
L(z):=\boldsymbol{h}^{\star}_{A_1\times\ldots\times A_N,\widehat{\mathbf{X}}_{N,k}}(z)=\\\max\Big\{\frac{1}{k}\sum\limits_{j=1}^N\boldsymbol{h}^{\star}_{A_j,D_j}(z_j),\max\limits_{j=1,\ldots,N}\{\boldsymbol{h}^{\star}_{A_j,D_j}(z_j)\}\Big\}=:R(z)
\end{multline*}
for every $z=(z_1,\ldots,z_N)\in\widehat{\mathbf{X}}_{N,k}.$ 
\end{prop}
\begin{proof}
Note that in view of the assumptions $\boldsymbol{h}^{\star}_{A_j,D_j}$ is continuous, $j=1,\ldots,N,$ and so is $L$ (the last observation follows from \cite{Si1}).\\ 
Observe that the inequality $L\geq R$ is obvious and $R$ is from the defining family for $L$, so we only need to prove the opposite inequality.\\
\indent We may assume that $D_j$ is relatively compact and strongly pseudoconvex, $j=1,\ldots,N$.\\
\indent We have $(dd^cL)^n=0$ on the set $\widehat{\mathbf{X}}_{N,k}\setminus(A_1\times\ldots\times A_N)=:V.$ Moreover, for any $z_0\in\partial V$ there is
$$
\liminf\limits_{V\ni z\rightarrow z_0}(R(z)-L(z))\geq 0,
$$
from which follows (in view of the domination principle)
$$
R\geq L \text{ on } \widehat{\mathbf{X}}_{N,k}\setminus(A_1\times\ldots\times A_N).
$$
Thus $R\geq L$ on $\widehat{\mathbf{X}}_{N,k}.$
\end{proof}
\begin{ex}
Let $0<r<R$ be two real numbers. Let $D_j=\mb{B}_{n_j}(R)\s\mb{C}^{n_j}$ be the euclidean open ball with center at $0$ and radius $R$ and let $A_j=\overline{\mb{B}_{n_j}}(r)$ be the euclidean closed ball with center at $0$ and radius $r$,\ $j=1,\ldots,N.$ Then $\boldsymbol{h}^{\star}_{A_j,D_j}=\max\left\{0,\frac{\log\frac{||\cdot||}{r}}{\log\frac{R}{r}}\right\}, j=1,\ldots,N$, and thus 
\begin{displaymath}
\boldsymbol{h}^{\star}_{A_1\times\ldots\times A_N,B}(z)=
\max\left\{\frac{1}{k}\sum_{j=1}^N\max\left\{0,\frac{\log\frac{||z_j||}{r}}{\log\frac{R}{r}}\right\},
\max\limits_{j=1,\ldots,N}\left\{0,\frac{\log\frac{||z_j||}{r}}{\log\frac{R}{r}}\right\}
\right\},
\end{displaymath}
for $z=(z_1,\ldots,z_N)\in B,$ where
\begin{displaymath}
B:=\left\{(z_1,\ldots,z_N)\in D_1\times\ldots\times D_N:\sum\limits_{j=1}^N
\max\left\{0,\frac{\log\frac{||z_j||}{r}}{\log\frac{R}{r}}\right\}<k\right\}.
\end{displaymath} 
\end{ex}
\begin{prop}[cf. \cite{J3}]
\label{envelope in envelope}
Let $1\leq k<l\leq N.$ Let $D_j\s\mb{C}^{n_j}$ be a pseudoconvex domain and let $A_j\s D_j$ be compact, locally pluriregular and locally $L$-regular, $j=1,\ldots,N.$ Then
\begin{multline*}
L(z):=\boldsymbol{h}^{\star}_{\widehat{\mathbf{X}}_{N,k},\widehat{\mathbf{X}}_{N,l}}(z)=\boldsymbol{h}^{\star}_{\mathbf{X}_{N,k},\widehat{\mathbf{X}}_{N,l}}(z)
=\\\max\left\{0,\frac{\sum\limits_{j=1}^N\boldsymbol{h}^{\star}_{A_j,D_j}(z_j)-k}{l-k}\right\}=:R(z)
\end{multline*}
for $z=(z_1,\ldots,z_N)\in\widehat{\mathbf{X}}_{N,l}.$
\end{prop}
\begin{proof} Observe that we only need to prove the third equality.
We may assume that $D_j$ is relatively compact and strongly pseudoconvex, $j=1,\ldots,N$.\\
\indent Observe that the inequality $\geq$ is obviuos, as the function $R$ belongs to the defining family for $L$. All we need to show is the opposite inequality.\\
For any $j=1,\ldots,N $ choose an increasing sequence $(K^s_j)_{s\in\mb{N}}$ of holomorphically convex locally $L$-regular compacta in $D_j$ containing $A_j$ and such that $\bigcup_{s=1}^{\infty}K^s_j=D_j$ (this is possible by the existence of an exhausting sequence of holomorphically convex compacta in $D_j$ and using \cite{Ze2}).\\
\indent Define 
\begin{displaymath}
L_s(z):=\boldsymbol{h}^{\star}_{\mathbf{X}^s_{N,k},\widehat{\mathbf{X}}_{N,l}}(z),
\end{displaymath}
$z\in \widehat{\mathbf{X}}_{N,l},$ where
$\mathbf{X}_{N,k}^s:=\mb{X}_{N,k}((A_j,{K_j^s})_{j=1}^N)$ (the $(N,k)$-crosses are defined for open $D_j$'s. However, in our context, the definition of $\mathbf{X}^s_{N,k}$ - formally the same as the definition of the $(N,k)$-cross - makes sense, as it is only set-theoretical).\\
\indent Notice that the functions $L_s$ are all continuous (see \cite{Si1}).\\
For a fixed $s\in\mb{N}$ there is 
\begin{displaymath}
L_s(z)\geq
\max\left\{0,\frac{\sum\limits_{j=1}^N\boldsymbol{h}^{\star}_{A_j,D_j}(z_j)-k}{l-k}\right\}.
\end{displaymath}
Furthermore, $(dd^c\boldsymbol{h}^{\star}_{\mathbf{X}^s_{N,k},\widehat{\mathbf{X}}_{N,l}})^n=0$ on $V_s:=\widehat{\mathbf{X}}_{N,l}\setminus\mathbf{X}^s_{N,k}$ and for any $z_0\in\partial V_s$ there is 
$$
\liminf\limits_{V_s\ni z\rightarrow z_0}(R(z)-L_s(z))\geq 0,
$$
from which follows that $R\equiv L_s,s\in\mb{N},$ and so $R\equiv L.$ 
\end{proof}

\end{document}